\documentclass[11pt,reqno]{amsart}
\usepackage[margin=3.6cm]{geometry}
\usepackage{graphicx}
\usepackage{amssymb}
\usepackage{bbm}
\usepackage{color}

\newcommand{\R}{{\mathbb R}}

\newcommand{\ep}{\varepsilon}

 \renewcommand{\Im} {\mathrm{Im\,}}

\newtheorem{theorem}{Theorem}

\newtheorem{proposition}[theorem]{Proposition}
\newtheorem{lemma}[theorem]{Lemma}

\theoremstyle{definition}
\newtheorem{definition}{Definition}
\newtheorem{remark}{Remark}

\def\reals{{\mathbb R}}

\def\p{\partial}

\def\tchi{\tilde{\chi}}
\def\Ci{{\mathcal C}^\infty}

\def\O{{\mathcal O}}

\begin{document}

\title[Small-scale mass of Neumann eigenfunctions]{Non-concentration estimates  for Laplace eigenfunctions on compact $C^{\infty}$ manifolds with boundary}

\author[H. Christianson]{Hans Christianson}
\address{Department of Mathematics, UNC Chapel Hill} \email{hans@math.unc.edu}

\author[J. Toth]{John  A. Toth}
\address{Department of Mathematics and
Statistics, McGill University, 805 Sherbrooke Str. West, Montr\'eal
QC H3A 2K6, Ca\-na\-da.} \email{jtoth@math.mcgill.ca}

\maketitle

\begin{abstract} 
Let $\Omega$ be an $n$-dimensional  compact Riemannian manifold  $(n \geq 3)$ with $C^\infty$ boundary, and consider $L^2$-normalized eigenfunctions $ - \Delta \phi_{\lambda} = \lambda^2 \phi_\lambda$ with  Dirichlet or Neumann boundary conditions .  

In this note, we extend well-known interior nonconcentration bounds up to  the boundary. Specifically, in Theorem \ref{thm1}, using purely stationary local methods, we prove that for such $\Omega$ it follows that for {\em any} $x_0 \in \overline{\Omega}$ (including boundary points) and for all $\mu \geq C_{\Omega} \lambda^{-1}$ with  sufficiently large constant $C_{\Omega} >0,$
\begin{equation} \label{nonconbdy}
 \| \phi_\lambda \|_{B(x_0,\mu)\cap \Omega}^2 = O(\mu). \end{equation}
 In Theorem \ref{thm2} we extend a result of Sogge \cite{So} to manifolds with smooth boundary and show that
 \begin{equation} \label{SUPBD}
 \| \phi_\lambda \|_{L^\infty(\Omega)} \leq C \lambda^{\frac{n}{2}} \cdot  \Big( \sup_{x \in \Omega} \| \phi_{\lambda} \|_{L^2( B(x,\lambda^{-1}) \cap \Omega )} \Big). \end{equation}
 The sharp sup bounds $\| \phi_{\lambda} \|_{L^\infty(\Omega)} = O(\lambda^{\frac{n-1}{2}})$ for Dirichlet or Neumann eigenfunctions proved by Grieser in \cite{Gr} are then an immediate consequence of Theorems \ref{thm1} and \ref{thm2}.

\end{abstract}\ \\

\section{introduction}
A fundamental issue regarding eigenfunctions involves their concentration properties  on small  balls with radius that depends on the eigenvalue $\lambda^2$ as $\lambda \to \infty.$  
As pointed  out in \cite{So}, if $\p \Omega = \emptyset$, using the explicit asymptotic formula for the half-wave operator $e^{it \sqrt{- \Delta}}: C^{\infty}(\Omega) \to C^{\infty}(\Omega)$ it is not hard to prove that there exists $C_\Omega>0$ such that 

\begin{equation} \label{soggebound}
\| \phi_\lambda \|_{L^2(B(r))}^2 = O(r) \| \phi_\lambda \|_{L^2(\Omega)}^2, \quad  \forall r \geq C_\Omega \lambda^{-1} \end{equation}\

We refer to estimates of the form (\ref{soggebound}) as {\em non-concentration} bounds.
The example of highest weight spherical harmonics  on the round sphere (see Remark \ref{gaussian} below) shows that (\ref{soggebound}) is, in general, sharp. However, in certain cases, one expects improvements. For instance, in the case of surfaces with non-positive curvature, one can get logarithmic improvements  \cite{So, Han}.

Since the wave parametrix on manifolds with boundary is quite complicated, it is of interest to give a  proof of (\ref{soggebound}) using purely stationary methods.
The first result of this paper in Theorem \ref{thm1} an
extension of the bounds in (\ref{soggebound}) to Dirichlet or  Neumann
eigenfunctions in the case the $\Omega$ is a compact manifold with smooth boundary. To avoid the usual  complications arising from the  behaviour of the wave parametrix near the boundary,  our basic arguments here use a  stationary $h$-microlocal factorization argument avoiding the wave propagator altogether.

It is useful at this point to switch to the convenient semiclassical scaling $h = \lambda^{-1}$ with  $-h^2 \Delta \phi_h = \phi_h$.  Moreover, without loss of generality, we assume in the following that $\phi_h$ is {\em real-valued}.
\begin{theorem} \label{thm1}
Let $(\Omega^n,g)$ be a compact, Riemannian manifold with $C^{\infty}$ boundary. Then, for any $p_0 \in \overline{\Omega} $ and  Laplace eigenfunction $\phi_h$ with eigenvalue $h^{-2},$ there exist constants $h_0>0$ and $C_{\Omega} >0$ such that  for all $h \in (0,h_0].$
\begin{equation} \label{one}
     \| \phi_h \|^2_{L^2 ( B(p_0 , \mu) \cap \overline{\Omega} )} = O(\mu); \quad \mu \geq C_{\Omega} h.
  \end{equation}
  \end{theorem}

\begin{remark} \label{gaussian}  The estimate in Theorem \ref{thm1} is sharp for all $\mu \geq h^{1/2}.$ To see this, one need only consider eigenfunctions that are Gaussian beams. For concreteness, let $ (S^2,g)$ be the round sphere. In terms of Euclidean coordinates $(x,y,z) \in \R^3$, the highest-weight spherical harmonics on $S^2 = \{(x,y,z) \in \R^3, x^2 + y^2 + z^2 =1 \}$ are given by
$$ \phi_h (x,y,z)  = (2\pi n)^{1/4} (x + i y)^{n}; \quad n=1,2,3,...$$
Setting $h = n^{-1}$ and noting that $ |\phi_h(x,y,z)| = (2\pi h)^{-1/4} ( 1- z^2 )^{1/h}$ when $(x,y,z) \in S^2,$

Assuming $\mu \geq h^{1/2},$ it follows that for any $p \in \{ (x,y,z) \in S^2, z =0 \},$

$$ \| \phi_h \|_{B(p,\mu)}^2  \approx C_1 h^{-1/2} \mu \, \int_{|z| < \mu } (1-z^2)^{2/h} \, dz $$
$$\approx C_2 \mu h^{-1/2} \int_{|z|<\mu} e^{-z^2/h} \, dz \approx \mu \int_{|w| < \mu h^{-1/2} }  e^{-w^2} dw \approx \mu,$$
and so, the non-concentration bounds are sharp when $\mu \in [h^{1/2},1].$

In the special case where the $\phi_h$ satisfy polynomial
 small-scale quantum ergodicity (SSQE) on a scales $\mu \geq h$ since the volume of a ball of radius $\mu$ is
$\mu^n$ one putatively expects a bound of $O(\mu^n)$ on the RHS in Theorem
\ref{thm1}. Unfortunately, to our knowledge, there are no
rigorous results  on polynomial SSQE  and seem well out of reach at present.
Logarithmic SSQE was proved by X. Han \cite{Han}. \end{remark}

\begin{remark} It is natural to expect that  the non-concentration bound in Theorem \ref{thm1} holds also when $n=2.$ Similarly, it should also hold under the weaker $C^2$-assumption on the boundary $\partial \Omega.$ We hope to address both points elsewhere. 
\end{remark}

\begin{remark}
In a companion paper \cite{CT}, using a different  Rellich-type commutator argument, we prove the standard non-concentration bounds in Theorem \ref{thm1}  in the special case of bounded, {\em convex}  Euclidean planar domains with corners under the constraint that
$\mu(h) \geq h^{1-\epsilon}$ for any fixed $\epsilon >0.$
\end{remark}

We thank Jeff Galkowsk and Michael Taylor for many helpful discussions regarding earlier versions of the paper.

\section{Proof of Theorem \ref{thm1}} 

\begin{proof}
We  first give the proof for {\em interior} balls with $B(x,\mu) \cap \partial \Omega = \emptyset$ and then adapt the proof to the case of {\em boundary} balls with $B(x,\mu) \cap \partial \Omega \neq \emptyset.$

 \subsection{Interior estimates on scales $\mu \geq C_{\Omega} h$} \label{int1} Here, we give an alternative stationary proof of the non-concentration bounds in the special case where either $(\Omega, g)$ is $C^{\infty}$ compact manifold without boundary, or alternatively, in the case where $\partial \Omega \neq \emptyset,$  with $ z_0 \in \mathring{\Omega}$ and $B(z_0,\mu) \in \mathring{\Omega}$ and the metric $g$ is assumed to be $C^{\infty}$ in $\mathring{\Omega}.$

Let $P(h): = - h^2 \Delta_g -1$ with principal symbol  
$p(x,\xi) = |\xi|^2_g - 1,$ where $|\xi|^2_g = \sum_{i,j =1}^n g^{ij}(x) \xi_i \xi_j.$  First, by energy concentration of eigenfunctions $\phi_h$ with $P(h) \phi_h =0,$ one can microlocalize to a fixed compact neighbourhood of the characteristic varietly $S^*M = p^{-1}(0).$ More precisely, given any open  interior neighbourhood $U \ni z_0$ and for  fixed $\epsilon >0,$ we set $S_{\ep}^*= \{ \xi \in T^*_x M;  | 1 - |\xi|_g^2 | \leq \epsilon. \}$ Then given any cutoff $\psi \in C^{\infty}_0(S^*_{\epsilon})$ with the property that $\psi |_{S_{\ep/2}^*} =1,$ by well-known energy concentration estimates \cite{Zw},

\begin{equation} \label{energycon}
\|(1-\psi(x, hD) ) \phi_h \|_{C^k(U)} = O_k(h^{\infty}).
\end{equation}

Then, given any cutoff $\rho \in C^{\infty}_0(U),$ we set  
$$v_h:= \rho(x) \psi(hD) \phi_h \in C^{\infty}_0(U)$$
and note that since $P(h) \phi_h = 0$ and $\rho(x) \psi(hD) \in Op_h(S^0),$
\begin{equation} \label{localize}
P(x,hD) v_h = [P(x,hD), \rho(x) \psi(hD)] \phi_h = O_{L^2}(h)
\end{equation} 
by $L^2$-boundedness. 

In the following, it suffices to work with the microlocalized eigenfunctions $v_h := \psi(x,hD) \phi_h.$ Since the real principal type condition $d_{\xi} \,p(x,\xi) \neq 0$ is satisfied for $(x,\xi) \in S^*_{\ep}U,$ it follows by an application of the implicit function theorem to $p(x,\xi)$ that there are points $w_j \in T^*U; j=1,..,N$ with open neighbourhoods $V_{w_j}; j=1,...N,$ such that
$$S_{\ep}^*U \subset \bigcup_{j=1}^N V_{w_j}$$
and with fiber coordinates $\xi = (\xi_k, \xi')$ on $V_{w_j},$
\begin{equation} \label{IFT}
p(x,\xi) = e_j(x,\xi) \big(  \xi_k - a_j(x,\xi') ); \quad e_j(x,\xi) \geq C_j>0, \,\,\, (x,\xi) \in V_{w_j}. \end{equation}\

Let $\rho_j \in C^{\infty}_0(T^*U); j=1,...,N$ be a partition of unity subordinate to the cover $\{ V_{w_j} \}_{j=1}^N$. It follows from (\ref{IFT}) and $L^2$-boundedness that with another cutoff $\tilde{\rho_j} \Supset \rho_j$ and $\tilde{\rho_j} \in C^{\infty}_0 (V_{w_j});[0,1]),$

\begin{equation} \label{IFT2}
\tilde{\rho_j} \,e_j(x,hD) \, \big(  hD_{x_k} - a_j(x,hD') \big) \rho_j v_h = O_{L^2}(h). \end{equation}\

In (\ref{IFT2}) and below,  we abuse notation somewhat and sometimes write $\rho_j:= \rho_j(x,hD)$ and $\tilde{\rho_j}: = \tilde{\rho_j}(x,hD)$ when the context is clear.

Since $e_j(x,hD)$ is $h$-elliptic on $V_{w_j}$ it follows by a standard $h$-microlocal parametrix construction applied to (\ref{IFT2}) that

\begin{equation} \label{IFT3}
\tilde{\rho_j} \big(  hD_{x_k} - a_j(x,hD') \big) \rho_j v_h = O_{L^2}(h). \end{equation}\

In the following, we will employ a number of convenient spatial cutoff
functions that we introduce here.

\vspace{1in}

  Let $\tchi(s) \in \Ci ( \reals)$ satisfy the following conditions:
  \begin{itemize}

  \item $\tchi$ is odd,

  \item $\tchi' \geq 0$,

    \item $\tchi(s) \equiv -1$ for $s \leq -3$ and $\tchi(s) \equiv 1$ for $s
      \geq 3$,

    \item
      $\tchi(-1) = -1/2$ and $\tchi(1) = 1/2$,

    \item $\tchi(s) = \frac{s}{2}$ for $-1 \leq s \leq 1$.

  \end{itemize}
See Figure \ref{F:tchi-2} for a picture. 
    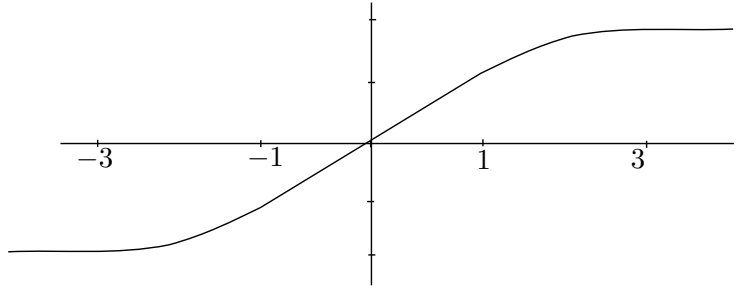
\begin{figure}
\hfill
\centerline{
\begingroup%
  \makeatletter%
  \providecommand\color[2][]{%
    \errmessage{(Inkscape) Color is used for the text in Inkscape, but the package 'color.sty' is not loaded}%
    \renewcommand\color[2][]{}%
  }%
  \providecommand\transparent[1]{%
    \errmessage{(Inkscape) Transparency is used (non-zero) for the text in Inkscape, but the package 'transparent.sty' is not loaded}%
    \renewcommand\transparent[1]{}%
  }%
  \providecommand\rotatebox[2]{#2}%
  \newcommand*\fsize{\dimexpr\f@size pt\relax}%
  \newcommand*\lineheight[1]{\fontsize{\fsize}{#1\fsize}\selectfont}%
  \ifx\svgwidth\undefined%
    \setlength{\unitlength}{273.47494125bp}%
    \ifx\svgscale\undefined%
      \relax%
    \else%
      \setlength{\unitlength}{\unitlength * \real{\svgscale}}%
    \fi%
  \else%
    \setlength{\unitlength}{\svgwidth}%
  \fi%
  \global\let\svgwidth\undefined%
  \global\let\svgscale\undefined%
  \makeatother%
  \begin{picture}(1,0.38988665)%
    \lineheight{1}%
    \setlength\tabcolsep{0pt}%
    \put(0,0){\includegraphics[width=\unitlength,page=1]{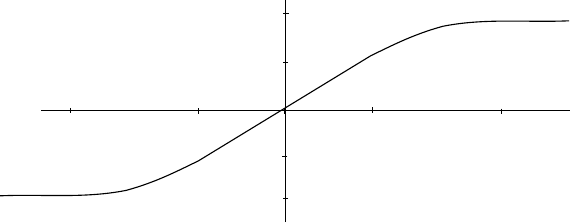}}%
    \put(0.09378875,0.16286937){\color[rgb]{0,0,0}\makebox(0,0)[lt]{\lineheight{1.25}\smash{\begin{tabular}[t]{l}$-3$\end{tabular}}}}%
    \put(0.32922772,0.16438834){\color[rgb]{0,0,0}\makebox(0,0)[lt]{\lineheight{1.25}\smash{\begin{tabular}[t]{l}$-1$\end{tabular}}}}%
    \put(0.64517173,0.15983148){\color[rgb]{0,0,0}\makebox(0,0)[lt]{\lineheight{1.25}\smash{\begin{tabular}[t]{l}$1$\end{tabular}}}}%
    \put(0.85782626,0.16135039){\color[rgb]{0,0,0}\makebox(0,0)[lt]{\lineheight{1.25}\smash{\begin{tabular}[t]{l}$3$\end{tabular}}}}%
  \end{picture}%
\endgroup%
}
\caption{\label{F:tchi-2} A  sketch of the function $\tchi$ used in
  the proof of Theorem \ref{thm1}.}
\hfill
\end{figure}
Let $\gamma(s) = \tchi'(s)$ so that $\gamma$ has support in $\{ -3
\leq s \leq 3 \}$, $\gamma(s) \geq 0$, and $\gamma(s) \equiv 1/2$ for
$| s | \leq 1$.

 In the following, it will also be  useful to define the corresponding rescaled functions
 $$\tilde{\chi}_{\mu}(s):= \tilde{\chi} ( \mu^{-1} s), \quad \gamma_{\mu}(s):= \gamma( \mu^{-1} s), \, \, \, \mu \geq h.$$

Next, setting $\gamma(x):= \prod_{k=1}^n \gamma(x_k),$ we multiply (\ref{IFT3}) by $\tilde{\chi}_{\mu}(x_k) \cdot \gamma(x)$ and integrate against $ \rho_j v_h(x)$ to get

\begin{align} \label{Int0}
\langle \tilde{\chi}_{\mu}(x_k) \gamma(x)  & \tilde{\rho_j}(x,hD) \, h D_{x_k}  \rho_j v_h,   \, \rho_j v_h \rangle_{L^2} \\
& = \langle \tilde{\chi}_{\mu}(x_k)  \gamma(x) \tilde{\rho_j}(x,hD) a_j(x,hD') \rho_j v_h,  \, \rho_j v_h \rangle_{L^2}  \nonumber 
    \,+ \, O(h).
\end{align} \

First, we note that since $\tilde{\rho_j}(x,\xi) =1 $ for $(x,\xi) \in \text{supp} \, \rho_j$, so the symbol of the commutator

$$ \sigma([ \tilde{\rho_j}, a_j(x,hD')])(x,\xi) \sim 0, \quad   (x,\xi) \in \text{supp} \, \rho_j.$$\

Consequently, from (\ref{Int0}) it follows that for the first term on the RHS of (\ref{Int0}),
\begin{align} \label{Int0.1}
 \langle \tilde{\chi}_{\mu}(x_k) \gamma(x)  \tilde{\rho_j} & a_j(x,hD') \rho_j v_h,  \, \rho_j v_h \rangle_{L^2} \\
 & = \langle \tilde{\chi}_{\mu}(x_k)  \gamma(x) a_j(x,hD') \rho_j v_h,  \, \rho_j v_h \rangle_{L^2}
    \,+ \, O(h^\infty). \nonumber
\end{align} \

Since $a_j(x,hD')^* = a_j(x,hD') + O(h)_{L^2 \to L^2},$  $[ a_j(x,hD'),  \tilde{\chi}_{\mu}(x_k) ] =0$ and by $L^2$-boundedness, 

$$[ a_j(x,hD'), \tilde{\chi}_{\mu}(x_k) \gamma(x) \tilde{\rho_j}(x,hD) ] = O(h)_{L^2 \to L^2},$$
it follows from (\ref{Int0.1}) that

\begin{equation} \label{imag1}
\Im   \langle \tilde{\chi}_{\mu}(x_k)  \gamma(x) \tilde{\rho_j}(x,hD) a_j(x,hD') \rho_j v_h,  \, \rho_j v_h \rangle_{L^2} = O(h). \end{equation} \

As for the LHS of (\ref{Int0}), we again use the fact that $\tilde{\rho_j} =1$ on  $ \text{supp} \, \rho_j$ and
$$ \sigma([ \tilde{\rho_j}, h D_{x_k} ])(x,\xi) \sim 0, \quad   (x,\xi) \in \text{supp} \, \rho_j.$$\
to get that

\begin{align} \label{imag2}
\langle \tilde{\chi}_{\mu}(x_k) \gamma(x) & \tilde{\rho_j}(x,hD) \, h D_{x_k}  \rho_j v_h,   \, \rho_j v_h \rangle_{L^2} \\
& = \langle \tilde{\chi}_{\mu}(x_k)  \gamma(x) \, h D_{x_k}  \rho_j v_h,   \, \rho_j v_h \rangle_{L^2}
    \,+ \, O(h).\nonumber
\end{align} \

Taking imaginary parts of both sides of (\ref{imag2}) gives

\begin{align} \label{imag3}
\Im \, \langle \tilde{\chi}_{\mu}(x_k) \gamma(x)  & \tilde{\rho_j}(x,hD) \, h D_{x_k}  \rho_j v_h,   \, \rho_j v_h \rangle_{L^2} 
 \\
 & = \frac{h}{2}  \idotsint \tilde{\chi}_{\mu}(x_k) \gamma(x)  \,\partial_{x_k}   \, | \rho_j v_h |^2 \, dx' dx_k \nonumber
    \,+ \, O(h).
\end{align} \

From (\ref{Int0}), (\ref{imag3}) and (\ref{imag1}) it then follows that

\begin{equation} \label{Int0.11}
\frac{h}{2} \idotsint    \tilde{\chi}_{\mu}(x_k) \gamma(x)   \, \partial_{x_k}  | \rho_j v_h(x) |^2  \, dx_k  dx' =  O(h).
\end{equation} \

Next, one integrates by parts on the LHS in the $x_k$-variable (\ref{Int0.11}). The boundary terms vanish since $\rho_j$ has compact support in the $x$-variables and differentiation of  $\gamma(x)$ term contributes an $O(h)$ term that can be absorbed in the RHS of (\ref{Int0.11}). The result is that

\begin{equation} \label{Int1}
\frac{h}{\mu} \idotsint    \gamma_{\mu}(x_k)  \gamma(x)  \, | \rho_j v_h(x) |^2  \, dx_k \, dx' =  O(h).
\end{equation} \

Since $\gamma_{\mu}(s) \in C^{\infty}_0(\R; [0,1])$ with $\gamma_{\mu}(s) \equiv \frac{1}{2} $ for $|s| \leq \mu,$ it follows from (\ref{Int1}) that

\begin{equation} \label{Int2}
 \frac{h}{\mu} \int_{|x_k| \leq \mu}  \int_{|x'| \leq 1}  |\rho_j v_h(x)|^2 \, dx_j dx' = O(h); \quad k=1,2,...,n.\end{equation}\
 
 Clearly the bound in (\ref{Int2}) holds for all indices $j=1,...N.$ Recalling that $ \sum_{j} \rho_j v_h = v_h + O_{C^{\infty}} (h^{\infty}),$  we use the pointwise bound
 $ |v_h|^2 \leq N  \sum_j |\rho_j v_h|^2 + O(h^{\infty}),$
 and cancel the $h$ on both sides of (\ref{Int2}) to get
 
 \begin{eqnarray} \label{Intupshot}
 \mu^{-1} \int_{|x_1| \leq \mu}  \cdots \int_{|x_n| \leq \mu}   |v_h(x)|^2 \, dx 
 \leq N  \mu^{-1} \sum_{j=1}^N \int_{|x_1| \leq  \mu} \cdots \int_{|x_n| \leq \mu}   |\rho_j v_h(x)|^2 \, dx \nonumber \\
 \leq N  \mu^{-1} \sum_{j=1}^N \int_{|x_k| \leq \mu}  \int_{|x'| \leq 1}   |\rho_j v_h(x)|^2 \, dx  = O(1).
  \end{eqnarray}\

 Since cubes of diameter $\mu$ are comparable to balls of radius $\mu$, the estimate in (\ref{Intupshot}) proves Theorem \ref{thm1} in the interior of $\Omega$ on {\em all} scales $0 < \mu \leq 1.$

 \subsection{Estimates at the boundary on scales $\mu \geq C_{\Omega} h.$}

 Let $\tilde{\Omega} \supset \overline{\Omega}$ be an open, smooth $n$-manifold that is an extension of $\overline{\Omega} = \Omega \cup \partial \Omega$ satisfying $\overline{\Omega} \subset \text{Int} \, \tilde{\Omega}.$ We make a $C^\infty$-extension of the Riemannian metric on $\overline{\Omega}$ to $\tilde{\Omega}$  and we abuse notation somewhat and denote the extended $C^{\infty}$ metric  by $g$ also.  We refer to $U:=\tilde{\Omega} \setminus \overline{\Omega}$ as the {\em collar extension} of $\Omega.$  Here, we assume that $B(z_0,\mu) \cap \partial \Omega$ and so,  we smoothly extend $\Omega$ to  $\tilde{\Omega}$ and smoothly extend the Riemannian metric to $\tilde{\Omega}.$ Unless specified otherwise,  $(x',x_n): U \to \R^n$ will  denote Fermi coordinates in $U$ with $\partial \Omega = \{ x_n = 0 \}.$  Since $\mu \to 0^+,$ without loss of generality we assume that  $B(z_0,\mu) \subset \tilde{\Omega}.$
 When $B(x,\mu) \cap \partial \Omega \neq \emptyset,$ we argue in a similar fashion to the interior case, but the $h$-microlocal factorization argument for $P(h) = - h^2 \Delta_g - 1$ is more subtle.

 Given $f \in C^{\infty}(\overline{\Omega}),$ we let $f^0$ be the extension of $f$ to $\tilde{\Omega}$ by zero. More precisely,
 \begin{equation} \label{extension}
 f^0(x) := 
 \begin{cases}
 f(x) \,\,\, \,\, x_n \geq 0,\\
 0\,\,\, \,\, x_n <0. \\
 \end{cases}
 \end{equation}
 
 Also, given $f \in C^{\infty}(\overline{\Omega}),$ we let $f_0:= f |_{\partial \Omega}$ denote the boundary trace.
 Following the argument in the interior case, we let $U \ni z_0$ be a sufficiently small open neighbourhood with $U \subset \tilde{\Omega}$ and let $\rho \in C^{\infty}_0(U).$ As in the interior case,  the first step is an energy localization (in this case, for $\phi_h^0$) that is crucial for the subsequent factorization argument. This turns out to be substantially weaker than in the interior due to boundary effects, but still suffices for non-concentration estimates.
 
 We summarize this in the following
 \begin{lemma} \label{bdyloc}
 For any fixed $\epsilon>0$, we let $S^*_{\ep} = \{ \xi \in T_x^*\tilde{\Omega}, \, | 1 - |\xi|_x | < \epsilon \} $ and consider a frequency cutoff $\psi \in C_0^{\infty}(S_{\ep}^*)$ with $\psi |_{S_{\ep/2} ^*} =1.$ Then, 
 $$\| (1 - \psi(x,hD) ) \phi_h^0 \|^2_{L^2(\tilde{\Omega})} = O(h^{\alpha}),$$
 with $\alpha =2$ (resp. $\alpha = \frac{4}{3}$)  in the Dirichlet (resp. Neumann) case.
 Here, we view $\psi(x,hD) \in \Psi^0_{sc}(\tilde{\Omega}).$
 \end{lemma}

 \begin{proof} We first apply the  operator $P(h) = -h^2 \Delta_g -1$ in Fermi coordinates to $\phi_h^0(x) = H(x_n) \phi_h(x)$ where $H(x_n) = 1$ when $x_n \geq 0$ and vanishes otherwise. Since $P(h) = h^2 D_{n}^2 + h a(x) hD_n + R(x,hD'),$ where $R$ is a tangential $h$-differential operator, direct computation gives

 \begin{eqnarray} \label{fermi}
\big( h^2 D_{n}^2 + h \,a(x) hD_n + R(x,hD') \big) ( H(x_n) \phi_h(x) ) \hspace{3in} \nonumber \\
 = ( P(h) \phi )^0(x)  - 2 h  ( \delta(x_n) \otimes h (\partial_n \phi)_0 ) - h^2 ( \delta'(x_n) \otimes \phi_0 ) 
 -i h^2 a(x) ( \delta(x_n) \otimes \phi_0 ) \hspace{.5in} \nonumber \\
 = - 2 h  ( \delta(x_n) \otimes (h \partial_n \phi)_0 ) - h^2 ( \delta'(x_n) \otimes \phi_0 ) -i h^2 a(x) ( \delta(x_n) \otimes \phi_0 ). \hspace{1in}
\end{eqnarray}
In the last step, we use that $(P \phi)^0 = 0.$
 Let $E(h) \in \Psi^0_{sc}(\tilde{\Omega})$ be an $h$-microlocal parametrix for $P(h)$ off $p^{-1}(0) = S^* \tilde{\Omega},$ so that with $L^2:= L^2(\tilde{\Omega}),$
 \begin{equation} \label{par}
 E(h) P(h) =  I -\psi(x,hD)  + O(h^{\infty})_{L^2 \to L^2}.
 \end{equation}
 An application of $E(h)$ to both sides of (\ref{fermi}) gives
 \begin{eqnarray} \label{fermi2}
 (I - \psi(x,hD)) \phi^0 + O_{L^2}(h^{\infty}) = -2 h  E(h)  ( \delta(x_n) \otimes h (\partial_n \phi)_0 ) - h^2 E(h) ( \delta'(x_n) \otimes \phi_0 ) \nonumber \\
 -i h^2 E(h) ( a(x) ( \delta(x_n) \otimes \phi_0) ). \hspace{2in}
 \end{eqnarray}

 We denote the semiclassical Neumann (resp. Dirichlet) boundary data by  $\phi^N(x',0) = h \partial_n \phi (x',0) = \delta(x_n) \otimes h (\partial_n \phi)_0$ and $\phi^D(x',0) = \delta(x_n) \otimes \phi_0 = \phi(x',0)$ respectively. 
 
 First, we assume Dirichlet boundary conditions with $\phi(x',0) \equiv  0.$ Then, 
 
 $$E(h) ( \delta'(x_n) \otimes \phi_0 )  =  E(h) ( a(x) ( \delta(x_n) \otimes \phi_0 ) ) =0$$
 
 in ${\mathcal D}'(\tilde{\Omega})$ and so, (\ref{fermi2}) simplifies to 
 
 \begin{equation} \label{D1}
 (I - \psi(x,hD)) \phi^0 = -2 h  E(h)  ( \delta(x_n) \otimes h (\partial_n \phi)_0 ) + O_{L^2(\tilde{\Omega}) }(h^{\infty}).
 \end{equation}

 Define
 \begin{equation} \label{T}
 T_D(h) \phi^{N} := - 2 h  E(h)  ( \delta(x_n) \otimes h (\partial_n \phi)_0 ).
 \end{equation}
 
  Then, $T_D(h): C^{\infty}(\partial \Omega) \to C^{\infty}(\overline{\Omega})$ is an $h$-Fourier integral operator ($h$-FIO) with Schwarz kernel of the form
 \begin{equation} \label{FIO1}
 T_D(x,y') = (2 \pi h)^{-n} \int_{\R^n} e^{ i x_n \xi_n / h} \, e^{i \langle x' -y', \xi' \rangle /h} \,  h \, e(x,\xi,h) \, d\xi' d\xi_n,
 \end{equation}
 where $e(x,\xi,h) \sim \sum_{j=0}^{\infty} e_j(x,\xi) h^{j}$ with principal symbol 
 $$e_0(x,\xi) = (1-\psi)(\xi) \cdot (|\xi|_x^2 -1)^{-1}$$
 and $\langle \xi \rangle^j e_j \in S^{0}_{sc}(T^*\tilde{\Omega}); j=1,2,3,....$
 \begin{remark} We note that there is an extra factor of $h$ appearing in the amplitude of (\ref{FIO1}) due to the $h$ multiplier on the RHS of (\ref{T}). \end{remark}
 Then, by a standard stationary phase argument in the $(x,\xi)$-variables, it follows that $T_D(h)^* T_D(h) \in h \Psi^{0}_{sc}(\partial \Omega)$ with

 \begin{equation} \label{FIO2}
  T_D(h)^* T_D(h)(y',z') = (2\pi h)^{-n} \int_{\R^{n-1}} e^{i \langle y'- z',\eta' \rangle/h}  h^2 \, e_{D}(y',\eta',h) \, d\eta',
  \end{equation}
 
 where $e_D \in S^{0}_{sc}(T^* \partial \Omega).$ Then, by $L^2$-boundedness, 
 $ \| T_D(h)^* T_D(h) \|_{L^2(\partial \Omega) \to L^2(\partial \Omega)} = O(h^2)$ and so,
 \begin{equation} \label{D2}
 \| T_D(h) \phi^N \|_{L^2(\overline{\Omega}) }   = O(h) \| \phi^N \|_{L^2(\partial \Omega)}  = O(h),
 \end{equation}
 where in the last estimate in (\ref{D2}) we have used that under Dirichlet boundary conditions \cite{HT, CHT},
 $$ \| \phi^N \|_{L^2(\partial \Omega)}^2 = \int_{\partial \Omega} | h \partial_{\nu} \phi |^2 d\sigma = O(1).$$
 Then, from (\ref{D1}) and (\ref{D2}) it follows that
 \begin{equation} \label{Dupshot}
 \| (I - \psi(x,hD)) \phi^0 \|^2_{L^2(\tilde{\Omega})} = O(h^2).
 \end{equation}

 Next, we assume Neumann boundary conditions $\partial_{n} \phi (x',0) \equiv 0.$
 Then, $  E(h)  ( \delta(x_n) \otimes h (\partial_n \phi)_0 ) = 0$ in ${\mathcal D}'(\tilde{\Omega})$ and so, from (\ref{fermi2})
 \begin{equation} \label{n0}
  (I - \psi(x,hD)) \phi^0 =  - h^2 E(h) \,( \delta'(x_n) \otimes \phi_0 ) + h^2 E(h) \, ( a(x) ( \delta(x_n) \otimes \phi_0) ) +O_{L^2}(h^{\infty}).
 \end{equation}
 We deal with the simplest term first. From (\ref{FIO1}) and (\ref{FIO2}), it follows that
 \begin{equation} \label{n1}
  \|  h^2 E(h) \, ( a(x) ( \delta(x_n) \otimes \phi_0 ) ) \|_{L^2(\tilde{\Omega})} = O(h^2) \| \phi^{D} \|_{L^2(\partial \Omega)} = O(h^{5/3}), \end{equation}
  
where in (\ref{N1}) we use the $L^2$-restriction bound $ \| \phi^{D} \|_{L^2(\partial \Omega)} = O(h^{-1/3})$ for Neumann eigenfunctions proved by Tataru \cite{Ta}.
Next, we deal with the more singular first term on the RHS  in (\ref{N1}). Write
\begin{equation} \label{n2}
 h^2 E(h) \,( \delta'(x_n) \otimes \phi ) =  h E(h) \,( h \delta'(x_n) \otimes \phi )
 \end{equation}
 and consider the operator $T_N: C^{\infty}(\partial \Omega) \to C^{\infty}(\overline{\Omega})$ given by 
 $$T_{N}(h) \phi^{D} := h E(h) \,( h\delta'(x_n) \otimes \phi).$$
 Writing the distributional integral $h \delta'(x_n) = (2\pi h)^{-1} \int_{\R} e^{i x_n \xi_n/h} \xi_n \, d\xi_n,$ it follows by an argument similar to the one in the Dirichlet case above that the Schwarz kernel of $T_N(h)$ is of the form
 \begin{equation} \label{n3}
 T_N(h)(x,y') = (2\pi h)^{-n} \int_{\R^n} e^{i x_n \xi_n/h} e^{i \langle x'-y',\xi'\rangle/h} \, h \, \tilde{e}(x,\xi,h) \, d\xi' d\xi_n,
 \end{equation}
 with $\tilde{e} \sim \sum_{j=0}^{\infty} \tilde{e}_j h^j$ where
 $$ \tilde{e}_0(x,\xi) = \xi_n \, (1-\psi)(\xi) \, ( |\xi|^2_x - 1)^{-1},$$
 and $\langle \xi \rangle^{j+1} \tilde{e}_j \in S^{0}_{sc}(T^*\tilde{\Omega})$ for all $j=0,1,2,....$ Then, by a  standard stationary phase argument similar to (\ref{FIO2}), $T_N(h)^* T_N(h) \in h^2 \Psi^{0}_{sc}(\partial \Omega),$ and so, by $L^2$-boundedness,
 \begin{equation} \label{n4}
 \| T_N(h) \phi^{D} \|_{L^2(\overline{\Omega})} = O(h) \| \phi^{D} \|_{L^2(\partial \Omega)} = O(h^{2/3}), \end{equation}
 again using that $\| u^D \|_{L^2(\partial \Omega)} = O(h^{-1/3}).$ Consequently, from (\ref{n0}), (\ref{n1}) and (\ref{n4}) it follows that in the Neumann case,
 \begin{equation} \label{Nupshot}
 \| (I- \psi(x,hD)) \phi^0 \|^2_{L^2(\tilde{\Omega})} = O(h^{4/3}).
 \end{equation}

 \end{proof}
 
 \subsubsection{Operator factorization.} To complete the proof of Theorem 1, given the energy localization in Lemma \ref{bdyloc} and  in analogy with the interior case, we $h$-microlocally factorize $P(h)$ in $S_{\epsilon}^*(\tilde{\Omega}).$ In the process, we adapt the resulting integration to the boundary case.
 Consider
 \begin{equation} \label{eigenfnloc}
 v_h^0(x):= \rho \,\psi(hD) \phi_h^0(x); \quad x \in U \subset \tilde{\Omega}.
 \end{equation}
 Then, from Lemma \ref{bdyloc}, it follows that
 \begin{equation} \label{fac1.1}
\| \rho \phi_h^0 - v_h^{0}\|^2_{L^2(U)} = O(h^{\alpha}), \,\,\, \alpha >1.\end{equation}
 
 Following the analysis in the interior case, for sufficiently small open set $U,$ we can cover
 $$S_{\epsilon}^*U \subset \bigcup_{j=1}^N  V_j,$$
 so that
 \begin{equation} \label{loc2}
 p(x,\xi) = e_j(x,\xi)  \, ( \xi_j - a_j(x,\xi') ), \quad e_j(x,\xi) \geq C_j >0, \, \, (x,\xi) \in V_j
 \end{equation}
 where $\xi' = (\xi_1,...,\hat{\xi_j},... \xi_n)$. In fact, one  can make the $V_j$'s more explicit. Setting $N=n$, for $\epsilon >0$ sufficiently small, let
 \begin{equation}\label{explicit}
 V_j = \big\{ (x,\xi) \in S_{\epsilon}^*U; \, |\xi_j| > \frac{1}{2} \big\}; \quad j=1,...,n.
 \end{equation}
 Note that on $V_j$, for the complimentary frequency coordinates $\xi_k; \, k\neq j,$ it follows that $|\xi_k| \leq \frac{1}{2} + \epsilon < 1.$ Let $\rho_j \in C^{\infty}_0(V_j); j=1,...,n$ be a corresponding partition of unity. By applying the $h$-microlocal parametrix for $e_j(x,hD)$ in $V_j$ and arguing as in (\ref{Int0.1}), (\ref{imag1}), (\ref{imag2}) and $(\ref{imag3}),$ it follows that

 \begin{equation} \label{fac1}
 \frac{h}{2} \int_{ [-3,3]^n \cap \overline{\Omega}} \tilde{\chi}_{\mu}(x_j) \gamma(x)   \, \partial_{x_j}  | \rho_j(hD) v^0_h(x) |^2  \, dx_j dx' =  O(h) + O(h^{\alpha}) = O(h).
\end{equation} \

We note that in view of Lemma \ref{bdyloc},  $\alpha >1$ and so, the $O(h^{\alpha})$-error in (\ref{fac1}) can be absorbed in the $O(h)$-error. 
Next, as in the interior case, we integrate by parts in (\ref{fac1}) but here, we must make sure to control for boundary terms. To do this, we treat {\em tangential} and {\em normal} differentiations to the boundary separately.
First, we deal with the tangential cases where $j \in \{ 1,...,n-1\}.$ Then, since $\gamma \in C^{\infty}_0 ([-3,3]^n)$ with $\gamma(x) = \prod_{k=1}^n \gamma(x_k),$
 
 \begin{eqnarray} \label{factan}
  \frac{h}{2} \int_{ [-3,3]^n \cap \overline{\Omega}} \tilde{\chi}_{\mu}(x_j) \gamma(x)   \, \partial_{x_j}  | \rho_j(hD) v^0_h(x) |^2  \, dx  \hspace{2in} \nonumber \\
 = \frac{h}{2} \int_{[0,3] \times [-3,3]^{n-2}} \prod_{k \neq j} \gamma(x_k)  \Big(  \int_{[-3,3]} \tilde{\chi}_{\mu}(x_j) \gamma(x_j)   \, \partial_{x_j}  | \rho_j(hD) v^0_h(x) |^2  \, dx_j  \Big) \, dx'. \end{eqnarray}
 where $[0,3]$ in (\ref{factan}) corresponds to integration in the normal variable $x_n$ and $\gamma(x_j) \in C^{\infty}_0([-3,3]).$ Integration by parts in the iterated  $x_j$-integral follows as before since there are no boundary terms.
 \
 In the normal case where $j=n,$
  \begin{eqnarray} \label{facnormal}
  \frac{h}{2} \int_{ [-3,3]^n \cap \overline{\Omega}} \tilde{\chi}_{\mu}(x_n) \gamma(x)   \, \partial_{x_n}  | \rho_nhD) v_h(x) |^2  \, dx  \hspace{2in} \nonumber \\
 = \frac{h}{2} \int_{[-3,3]^{n-1}} \prod_{k \neq n } \gamma(x_k)  \Big(  \int_{[0,3]} \tilde{\chi}_{\mu}(x_n) \gamma(x_n)   \, \partial_{x_n}  | \rho_n(hD) v_h(x) |^2  \, dx_n \Big) \, dx' \nonumber \\
 \frac{h}{2\mu} \int_{[-3,3]^{n-1}} \prod_{k \neq n } \gamma(x_k)  \Big(  \int_{[0,3]} \tilde{\chi}_{\mu}'(x_n) \gamma(x_n)   \,  | \rho_n(hD) v_h(x) |^2  \, dx_n \Big) \, dx' \nonumber \\
 + \frac{h}{2} \int_{[-3,3]^{n-1}} \prod_{k \neq n } \gamma(x_k)   \tilde{\chi}_{\mu}(3) \gamma(3)   \, | \rho_n(hD) v_h^0(x',3) |^2   \, dx' \nonumber \\
  - \frac{h}{2} \int_{[-3,3]^{n-1}} \prod_{k \neq n } \gamma(x_k)   \tilde{\chi}_{\mu}(0) \gamma(0)   \, | \rho_n(hD) v_h^0(x',0) |^2   \, dx'.
  \end{eqnarray}
  
 However, since  $ \gamma(3) = \tilde{\chi}_{\mu}(0)=0,$ the last two boundary terms in (\ref{facnormal}) both vanish and the rest of the proof then follows as in the interior case.

 \end{proof}

 \section{Non-concentration and $L^\infty$ bounds for eigenfunctions}

 In the case of $C^{\infty}$ compact manifolds without boundary, the eigenfunction sup bounds
 \begin{equation} \label{SUP}
 \| \phi_h \|_{L^\infty(\Omega)}= O(h^{\frac{1-n}{2}})
 \end{equation}
 were proved by Avakumovic \cite{Av} and Levitan \cite{L} in the case of Laplace eigenfunctions and the bound was proved for eigenfunctions of general self-adjoint elliptic operators by H\"{o}rmander \cite{Ho}.
 For Dirichlet or Neumann eigenfunctions on $C^\infty$ manifolds with boundary,  the  analogous bound was proved by Grieser \cite{Gr} using wave methods. 
 In this section, we use Theorem \ref{thm1} to give an alternative entirely local and stationary proof of the eigenfunction sup bounds (\ref{SUP}) for Dirichlet or Neumann eigenfunctions in the case of manifolds with $C^\infty$ boundary.
 
 First, we note that for any $p \in \Omega,$ from the eigenfunction sup bounds $ \| \phi_h  \|_{L^\infty(\Omega)} = O(h^{\frac{1-n}{2} })$ it is immediate that
 $$  \int_{B(p_0,h)} | \phi_h|^2 dvol \leq C_n h^{n} h^{1-n} = C_n h.$$
 
 Thus, on the fundamental wave length scale $\mu = h$, the non-concentration bound in Theorem \ref{thm1} follows immediately from the eigenfunction sup bounds.

 We note that there is actually a converse to this result. In the case where $\partial \Omega = \emptyset,$  the result below was proved by Sogge in \cite{So} (see (3.3) on pg. 391) using wave methods. We extend this result to Dirichlet or Neumann eigenfunctions on manifolds with boundary using a  local, stationary argument.
 
 \begin{theorem} \label{thm2} Let $(\Omega^n,g)$ be a compact $C^\infty$ Riemannian manifold  $ (n \geq 3) $ with or without boundary and $\phi_h$ be an $L^2$-normalized Laplace eigenfunction satisfying Dirichlet or Neumann boundary conditions in the case where $\partial \Omega \neq \emptyset.$ Then, there exist a uniform constant $C = C(\Omega^n,g)$ such that
 $$\| \phi_h \|_{L^{\infty}(\Omega)} \leq C  h^{- \frac{n}{2}}  \,  \Big( \sup_{x \in \Omega}  \| \phi_h \|_{L^2(B(x,h) \cap \Omega)} \Big)$$
 \end{theorem}

 \begin{proof}
 
 We first prove this outside an $h$-width Fermi neignbourhood of the boundary given by
 $$ \Omega_{int}(h) := \{ x \in \Omega;  d(x,\partial \Omega) \geq h \}.$$

\subsection{Sup bounds at interior points} \label{interior points} We first assume here that  $p_0 \in \Omega_{int}(h)$  and let $x: B \to \R^n$ be geodesic normal coordinates defined in an open ball $B = B(p_0,\ep_0)$with $x(p_0) = 0.$ In the following, we let $\Delta_0 = \sum_{j} \partial_{x_j}^2$ be the local flat Laplacian in the $x$-coordinates and note that in the small ball $B_h = \{ |x| \leq h \},$ there exists a second-order differential operator $A(x,D_x)$ with
 
 \begin{eqnarray} \label{laplaceapprox}
 \Delta \phi = \Delta_0 \phi + A(x,D_x)\phi,  \hspace{1in} \nonumber \\
 \| A(x,D_x) \phi \|_{L^2(B_h)} = O(h^2) \| \phi \|_{H^2(B_h)} + O(h) \| \phi \|_{H^1(B_h)},
 \end{eqnarray}
 where the last line in (\ref{laplaceapprox}) follows from the fact that $g^{ij}(x) = \delta_j^i + O(|x|^2)$ with $\nabla g^{ij}(x) = O(|x|).$

 From standard elliptic estimates

 \begin{equation} \label{ELLIPTIC}
 \| \phi \|_{H^{s+2}(B_h)} \leq C_1 \| (- \Delta + I) \phi \|_{H^s(B_{2h})} + C_2 \| \phi \|_{L^2(B_{2h})} \end{equation}\
  and the non-concentration bound $\| \phi \|_{B_{2h}}^2 = O(h)$
 it then follows from (\ref{laplaceapprox}) that
 
 \begin{equation} \label{laplaceapprox2}
 \| A(x, D_x) \phi \|_{L^2(B_h)}  = O(1) \| \phi_h \|_{L^2(B_{2h})} = O(h^{1/2}).
 \end{equation}\

Assume  that $n \geq 3.$ Given $x_0 \in B$ we 
 set $B_h(x_0):= \{ y; |y-x_0|  \leq  h \}$ and $S_h(x_0):= \{ y; |y-x_0| = h \}$ in the following.  Since $\phi$ is approximately harmonic on scale $\leq h,$ we mimic the proof of the mean value theorem to derive the required a priori estimates. 
An application of Green's formula with $\phi(y)$ and the local Green's function $G(x_0,y) = c_n |x_0 -y|^{2-n}$  with $(\Delta_0)_y G(x_0,y) = \delta_{x_0}(y)$ gives

 \begin{eqnarray} \label{MVT1}
 \phi(x_0) - \int_{B_h(x_0)} \Delta_0 \phi (y) \,\, G(x_0,y) \, dy \nonumber \\
 = c_n' h^{1-n}  \int_{S_h(x_0)} \phi(y) d\sigma(y) - c_n h^{2-n} \int_{S_h(x_0)} \partial_r \phi(y) d\sigma(y),
 \end{eqnarray}
 where $c_n' = (2-n) c_n$ and $r= |y-x_0|.$

  Substitution of (\ref{laplaceapprox}) and (\ref{laplaceapprox2}) in the LHS of (\ref{MVT1}) gives  
 
 \begin{eqnarray} \label{MVT1.1}
 \phi(x_0) + c_n h^{-2} \int_{B_h(x_0)}  \phi (y) |x_0-y|^{2-n} dy + c_n \int_{B_h(x_0)}  A \phi (y) |x_0 - y|^{2-n} \, dy \nonumber\\
 = c_n' h^{1-n} \int_{S_h(x_0)} \phi(y) d\sigma(y) - c_n h^{2-n} \int_{S_h(x_0)} \partial_r \phi(y) d\sigma(y), 
 \end{eqnarray}

 By another application of Green and (\ref{laplaceapprox2}), 
  \begin{eqnarray} \label{MVT2}
 - \int_{S_h(x_0)} \partial_{r} \phi(y) d\sigma(y) = - \int_{B_h(x_0)} \Delta \phi (y) dy + \int_{B_h(x_0)} A\phi (y) dy \nonumber \\
 = h^{-2} \int_{B_h(x_0)} \phi(y) dy  + O(h^{\frac{n}{2}} \| \phi_h \|_{L^2(B_{2h})}). \end{eqnarray}

 Substitution of (\ref{MVT2}) in (\ref{MVT1.1}) gives
 
  \begin{eqnarray} \label{MVT3}
 \phi(x_0) + c_n h^{-2} \int_{B_h(x_0)} \phi(y) |x_0-y|^{2-n} dy + c_n \int_{B_h(x_0)}  A \phi (y) |x_0 - y|^{2-n} \, dy\nonumber \\
 = c_n' h^{1-n} \int_{S_h(x_0)} \phi(y) d\sigma(y) + c_n h^{2-n} h^{-2} \int_{B_h(x_0)} \phi(y) dy + O(h^{2-\frac{n}{2}} \| \phi_h \|_{L^2(B_{2h})}).
 \end{eqnarray}

 We rewrite the first term on the RHS of (\ref{MVT3}) as a ball integral using Green yet again with functions $\phi(y)$ and $\rho(y) = |y-x_0|^2.$ This gives  
 $$ - h^{-2} \int_{B_h}  \phi \, \rho  dy - 2 \int_{B_h}  \phi dy  = h^2 \int_{S_h} \partial_{r} \phi d \sigma(y) - 2 h \int_{S_h} \phi d \sigma (y) $$

 $$ \hspace{2in}= - h^2 h^{-2}  \int_{B_h} \phi dy - 2 h  \int_{S_h}  \phi d \sigma(y) + O(h^2 h^{ \frac{n}{2}} \| \phi_h \|_{L^2(B_{2h})}). $$
 where in the last line we have input (\ref{MVT2}).

Solving for $h^{1-n} \int_{S_h} \phi \, d \sigma$ gives:

 \begin{eqnarray} \label{sphere int}
 2h^{1-n} \int_{S_h(x_0)}  \phi  d \sigma  = h^{-2-n}\, \int_{B_h(x_0)} \phi  \cdot  \rho dy  + 2 h^{-n} \,  \, \int_{B_h(x_0)} \phi dy \nonumber \\
  - h^{-n} \int_{B_h(x_0)} \phi dy + O(h^{-\frac{n}{2}} h^2  \| \phi_h \|_{L^2(B_{2h})}). 
 \end{eqnarray}
 
 Finally, substitution of (\ref{sphere int}) in (\ref{MVT3}) gives

 \begin{eqnarray} \label{UPSHOT}
 \phi(x_0) =  -c_n h^{-2} \int_{B_h(x_0)} \phi(y) |x_0-y|^{2-n} dy - c_n\int_{B_h(x_0)} A \phi (y) |x_0-y|^{2-n} \, dy \nonumber \\
  +  \frac{c_n'}{2} h^{-2-n} \, \int_{B_h(x_0)} \phi(y) \cdot  \rho(y) \, dy  -  c_n'  \frac{h^{-n}}{2} \int_{B_h(x_0)} \phi(y) dy  +\nonumber \\
    + c_n' h^{-n} \int_{B_h(x_0)}  \phi(y) dy + c_n h^{-n} \int_{B_h(x_0)} \phi(y) dy + O(h^{-\frac{n}{2}} h^2  \| \phi_h \|_{L^2(B_{2h})}). 
 \end{eqnarray} \\
  
  To estimate the first integral on the RHS, we introduce polar variables with $r(y) = |x_0-y|$ and integrate by parts $n-3$ times with respect to $\partial_r.$ More concretely, let $\chi \in C^{\infty}_0(B_1)$ be a radial cutoff with $\chi= \chi(r)$  satisfying  $\chi |_{B_{1/2}} = 1$ and $0 \leq \chi \leq 1.$

Then, setting  $I_1(h):=  h^{-2}  \int_{B_h(x_0)} \phi(y) |x_0 -y|^{2-n} dy,$  we write

\begin{equation} \label{first term}
 I_1(h) = h^{-2} \int_{B_h} |y-x_0|^{2-n} \phi_h(y) \, \chi(h^{-1}r) dy   +  h^{-2} \int_{B_h} |y-x_0|^{2-n} \phi_h(y) \, (1-\chi)(h^{-1}r) dy        \end{equation}

 To bound the second term on the RHS of (\ref{first term}), we apply Cauchy-Schwarz to get
 
 \begin{eqnarray} \label{first term.2}
h^{-2} \Big|   \int_{B_h} |y-x_0|^{2-n} \phi_h(y) \, (1-\chi)(h^{-1}r) dy  \Big| \leq C_n  h^{-n} \| \phi_h \|_{L^1(B_h)} \nonumber \\
\leq C_n h^{-n} | B_h|^{1/2} \, \| \phi \|_{L^2(B_h)} \leq C_n h^{-\frac{n}{2}} \| \phi \|_{L^2(B_h)}. \end{eqnarray}

As for the first term on the RHS of (\ref{first term}) we note that $f(y) = |y-x_0|^{2-n} \notin L^2_{loc}$ when $n >3.$ To remedy this, we pass to polar coordinates with $r(y) = |y-x_0|$ and  integrate by parts $(n-2)$-times with respect to $\partial_r.$ Since boundary term vanishes at $r=0,$

\begin{equation} \label{ibp}
  h^{-2} \int_{B_h} |y-x_0|^{2-n} \phi_h(y) \, \chi(h^{-1}r) dy = C_n h^{-2} \int_{B_h} \, \partial_r^{n-2} \big(   \phi_h(y) \chi(h^{-1} r) \big) \, dy,
 \end{equation}
and  by Cauchy-Schwarz,

\begin{equation} \label{parts}
 \Big| C_n h^{-2} \int_{B_h} \, \partial_r^{n-2} \big(   \phi_h(y) \chi(h^{-1} y) \big) \, dy \Big| \leq C_n h^{-2}  \Big( \sum_{k=0}^{n-2} h^{-k} \| \phi \|_{H^{n-2-k}(B_h)} \Big) \, | B_h |^{\frac{1}{2}}. \end{equation}

From the elliptic estimates (\ref{ELLIPTIC}), $  \| \phi \|_{H^{n-2-k}(B_h)} = O(h^{-(n-2-k})) \| \phi \|_{L^2(B_{2h})} $ and so,
from (\ref{parts}) it follows by another application of Cauchy-Schwarz that
\begin{align} \label{first term.1}
\Big| h^{-2} \int_{B_h} |y-x_0|^{2-n} \phi_h(y) \, \chi(h^{-1}y) dy \Big| & \leq C_n h^{-2} h^{2-n} \, | B_h |^{\frac{1}{2}} \, \| \phi \|_{L^2(B_{2h}) } \\
&= O(h^{-\frac{n}{2}})  \| \phi \|_{L^2(B_{2h}) }. \nonumber
\end{align}

Consequently, from (\ref{first term.1}) and (\ref{first term.2}),  $I_1(h) = O(h^{-\frac{n}{2}}  \| \phi \|_{L^2(B_{2h}) } ).$

 The second integral is estimated in the same way using integration by parts  with respect to $\partial_{r}^{n-2}$, Cauchy -Schwarz and the  elliptic estimates
 
 $$  \| A \phi \|_{H^{n-2-k}(B_h)} = \O( h^{2}) \| \phi \|_{H^{n-k}(B_h)} = O(h^{-(n-2-k}) \| \phi \|_{L^2(B_{2h})}. $$

The result is that
\begin{eqnarray} \label{error}
\Big|   \int_{B_h} A \phi (y) |x_0-y|^{2-n} dy  \Big| \leq C_n   \Big( \sum_{k=0}^{n-2} h^{-k} \| A\phi \|_{H^{n-2-k}(B_h)} \Big) \, | B_h |^{\frac{1}{2}}  \nonumber \\
 \leq C_n h^{2-n}  |B_h|^{\frac{1}{2}}  \| \phi \|_{L^2(B_{2h}) } = O(h^{-\frac{n}{2}} h^2  \| \phi \|_{L^2(B_{2h}) } ). \end{eqnarray}

The rest of the terms  on the RHS of (\ref{UPSHOT}) are all estimated in a similar fashion. By Cauchy-Schwarz and Theorem \ref{thm1},

$$  h^{-n} \Big| \int_{B_h(x_0)} \phi(y) dy  \Big| \leq C_h h^{-n}  \, | B_h(x_0)|^{1/2} \, \| \phi \|_{L^2(B_h)}$$
$$ \leq C_n h^{-n} h^{\frac{n}{2} }  \| \phi \|_{L^2(B_{2h}) } .$$

In addtion, since $ \max_{y \in B_h} \rho(y) \leq h^2,$

$$ h^{-2-n}\, \int_{B_h(x_0)} \phi \cdot  \rho \, dy =   O(h^{-\frac{n}{2}}  \| \phi \|_{L^2(B_{2h}) } ).$$

The upshot is that $| \phi(x_0) |  = O(h^{-\frac{n}{2}}   \| \phi \|_{L^2(B_{2h}) } )$ for any $x_0 \in \Omega_{Int}(h)$ and since this bound is uniform, it follows that

\begin{equation} \label{intsupbounds}
\sup_{x  \in \Omega_{Int}(h)}  |\phi(x)|  \leq C_n h^{-\frac{n}{2} } \, \sup_{x \in \Omega_{Int}(h)} \| \phi \|_{L^2(B_{h}(x)) } = O(h^{\frac{1-n}{2} }).
\end{equation}

\subsection{Sup bounds at  boundary points $x \in \partial \Omega.$}  Since in the Dirichlet case, $\phi |_{\partial \Omega} = 0$, only the Neumann case is of interest here. 
We let $(y',y_n): U \to \R^n$ denote Fermi coordinates in  a collar neighbourhhod $U \subset \tilde{\Omega}$ of $\partial \Omega = \{ y_n = 0 \}.$ Given $x \in \partial \Omega,$ we consider the set
$$B_h^+ := \{ y \in \overline{\Omega}; \, r(x,y) \leq h, \, y_n \geq 0 \} $$ 
written in Fermi coordinates.

 For convenience, in the following we denote the reflection map $\iota: B^{\pm} \to B^{\mp}$ by $\iota(y',y_n) = (y',-y_n)=: y^*.$ We note that under the assumption here that $x \in \partial \Omega,$ the reflection map $y \to y^*$ maps $B_{h}^{\pm} \to B_{h}^{\mp}$ where $B_h = B_h^+ \cup B_h^- $ is contained in the extension $\tilde{\Omega} \supset \Omega$ and the metric $g$ is extended smoothly across the boundary. In particular, we note that 
 $$r(x,y) = r(x,y^*), \,\, \forall x \in \partial \Omega.$$

\begin{remark} We note that in the above the metric extensions to $\tilde{\Omega}$ do {\em not} involve reflection in the boundary to avoid the usual issues with only Lipschitz regularity of the metric. Indeed, here we simply extend $g$ in a smooth fashion across the boundary $\partial \Omega.$ and then denote the corresponding Fermi coordinates by $(y',y_n).$
\end{remark}

Since by assumption $B_h \cap \partial \Omega \neq \emptyset$, in this case it is more convenient to work with the invariant Laplacian in Fermi coordinates in $B_h^+$ rather than passing to the flat Laplacian in geodesic normal ccordinates, as in the interior case. Given $x \in \partial \Omega(h),$
to estimate $|\phi(x)|$ we apply Green's formula with the Neumann Green's function $G_N(x,y).$  Before we do this, for the benefit of the reader, we begin with some preliminaries on the local asymptotics of the  Dirichlet (resp. Neumann) Green's functions $G_{D}(x,y)$ (resp. $G_N(x,y).$ This material is well-known to experts but is difficult to find in the literature in the form we will need here, so we give the details below.

\subsubsection{Local near-diagonal asymptotics of $G_{D,N}$}
Given $x \in M,$ let  $U \subset M$ be open set with $ d(U,x)  < inj(M,g)$ and  consider the local free Green's function $G(x,\cdot) \in {\mathcal D}'(U)$ satisfying
$$ -\Delta_y G(x,y) = \delta_x(y), \quad y \in U.$$
First, let $E \in \Psi^{-2}_{cl}$ be a  parametrix for $\Delta$ with
$$\Delta \cdot E = I + K, \quad K \in \Psi^{-\infty}.$$
Then, since, in particular,  $K(x,y) \in C^{\infty}(U\times U),$ it follows that for $U$ sufficiently small, $ \| K \|_{L^2(U) \to L^2(U)} \leq \frac{1}{2}$ and so, $(I + K)^{-1}$ has a local norm-convergent Neumann series expansion with
$ (I + K)^{-1} = \sum_{j=0}^{\infty} (-1)^j K^j $ as operators from $L^2(U)$ to $L^2(U).$ Consequently, it follows that for $U$ small,
\begin{equation} \label{locgreen}
G = E (I + K)^{-1} = E + K'; \quad K' \in \Psi^{-\infty}(U).
\end{equation}
To write $G(x,y)$ locally near the diagonal $x=y,$  modulo smoothing operators it is  enough  to write the near-diagonal conormal expansion in powers of $r(x,y)$  for the kernel of the parametrix $E(x,y)$ with $(x,y) \in U \times U. $ Modulo $R(x,y) \in C^{\infty}(U \times U),$ one can write this as a distributional integral

\begin{equation} \label{green1}
E(x,y) = (2\pi)^{-n} \int_{T^*_x U } e^{-i \exp_x^{-1}(y) \cdot \xi} \, a(x,\xi) \, d\xi,
\end{equation}
where, $a \sim \sum_{j=0}^{\infty} a_j$ where $a_j \in S^{-j}_{cl}(T^*U)$ and $a_0(x,\xi) = |\xi|_x^{-2}.$

Making a change to polar coordinates in (\ref{green1}) with $\rho = |\xi|_x$ and $\omega \in S_x^*,$  it follows that the distributional integral in (\ref{green1}) can be written in the form
\begin{equation} \label{green2}
G_0(x,y) = \int_{0}^{\infty}  \int_{S_x^*} e^{i \rho \, r(x,y) \langle \omega_0,\omega \rangle} \, d\omega_x \rho^{n-3}  d\rho, \end{equation}
where $\omega_0(x,y):= \frac{ \exp_x^{-1}(y) }{r(x,y)}.$
We make the radial change of coordinates $\rho \mapsto \rho \, r(x,y) =s$ in (\ref{green2}) to get
\begin{equation} \label{green3}
G_0(x,y) = r(x,y)^{2-n}  \int_{-\infty}^{\infty} \Big( \int_{S^*_x} e^{is \langle \omega_0, \omega \rangle } \, d\omega_x \Big) s_{+}^{n-3}  \, ds. \,  \end{equation}
Let $\chi_{+} (u) \in C^{\infty}(\R)$ equal $1$ in $[1,\infty )$ with supp $\chi_+ \subset [1/2,\infty),\, \chi_{-} \in C^{\infty}(\R)$ with supp $\chi_{-} \subset (-\infty, -1/2]$ and $\chi_{-}$ equal to $1$ on $(-\infty, - 1].$ Finally, we  choose $\chi \in C^{\infty}_0((-1,1))$ with $\chi |_{[-1/2,1/2]} = 1$ and so  that  $\chi_{-1} + \chi_{+} + \chi \equiv 1.$
 We split the integrand in (\ref{green3}) into the sum
\begin{eqnarray} \label{split}
 \int_{0}^{\infty} \int_{S_x^*} e^{i s \langle \omega_0, \omega \rangle } \, \chi_{-}( \langle \omega_0, \omega \rangle ) \, d\omega_x  s_+^{n-3} ds  +  \int_{0}^{\infty} \int_{S_x^*}  e^{i s \langle \omega_0, \omega \rangle } \, \chi_+( \langle \omega_0, \omega \rangle ) d\omega_x \, s_+^{n-3} ds \nonumber \\
+  \int_{0}^{\infty} \int_{S_x^*}  e^{ is \langle \omega_0, \omega \rangle } \,  \chi( \langle \omega_0, \omega \rangle )  \, d\omega_x s_+^{n-3} ds=: F_+(x) + F_{-}(x) + F(x).
\end{eqnarray}
Next, we write
$$F_{\pm}(x) = \lim_{\epsilon \to 0^+} \int_{0}^{\infty} \int_{S_x^*} e^{(\mp \epsilon + i) s   \langle \omega_0, \omega \rangle } \, \chi_{\pm}( \langle \omega_0, \omega \rangle ) \, d\omega_x s_+^{n-3} ds$$
$$= c_n \int_{S_x^+} \chi_{\pm}( \langle \omega_0, \omega \rangle ) \, |\langle \omega_0, \omega \rangle |^{2-n} \, d\omega_x, $$

where the last line  above follows by a contour deformation argument as in \cite{Ho2}(see Example 7.1.17 on pg. 167).

 Finally, since the spherical measure $d\omega_x$ is invariant under rotation, in the above $\omega_0 = \frac{ exp_x^{-1}y}{ r(x,y)} $ can be replaced by a fixed point (say $\theta_0 = (0,...,0,1) )$ in $ S_{x}^*.$ 

Consequently,
$$F_{\pm}(x) = c_n \int_{S_x^+} \chi_{\pm}( \langle \theta_0, \omega \rangle ) \, |\langle \theta_0, \omega \rangle |^{2-n} \, d\omega_x, $$
and the fact that $F_{\pm} \in C^{\infty}(U)$ follows since $d\omega_x$ is locally $C^{\infty}$ in $x.$

Finally, we write
\begin{eqnarray} \label{F}
F(x) = \int_{0}^{1}  \int_{S_x^*}  e^{ i s \langle \theta_0, \omega \rangle } \,  \chi( \langle \theta_0, \omega \rangle )  \, d\omega_x s_+^{n-3} ds +  \int_{1}^{\infty} \int_{S_x^*}  e^{i s \langle \theta_0, \omega \rangle } \,  \chi( \langle \theta_0, \omega \rangle )  \, d\omega_x s_+^{n-3} ds. \end{eqnarray}
The first integral is in $C^{\infty}(U)$ by differentiation under the integral sign since $d\omega_x$ is smooth for $x \in U.$ In the second integral on the RHS of (\ref{F}), we note that  the phase function $\Phi \in C^{\infty}(S_x^*)$ given by
$$ \Phi(\omega) = \langle \theta_0, \omega \rangle, $$
has critical points at $\omega_c = \pm \omega_0$ where $\langle \omega_c, \omega_0 \rangle = \pm 1$ and so,

$$ \nabla_{\omega} \Phi(\omega) \neq 0, \,\,\,\text{when} \,\, \,  \langle \theta_0, \omega \rangle \in \text{supp} \, \chi.$$
So, by repeated integration by parts in $\omega$, this term is of the form
$$  \int_{1}^{\infty} \int_{S_x^*}  e^{i s \langle \theta_0, \omega \rangle } \,  \chi( \langle \theta_0, \omega \rangle )  \,{\mathcal O}_N(s^{-N})\, d\omega_x   ds$$
 for any $N >1$ and this  is also in $C^{\infty}(U).$

The upshot is that (\ref{green3}) can be written in the form

\begin{eqnarray} \label{green4}
G_0(x,y) = c_n r(x,y)^{2-n} F_0(x)\end{eqnarray}
where $F_0 \in C^{\infty}(U).$

To analyze the lower-order terms  in $E(x,y)$ in (\ref{green1}) given by
$$G_j(x,y):=  (2\pi)^{-n} \int_{T^*_x U } e^{-i \exp_x^{-1}(y) \cdot \xi} \, a_j(x,\xi) \, d\xi, \,\,\, j=1,2,3,...$$
we argue as for the $G_0$-term above noting that $a_j(x,\xi )$ is positive homogeneus of degree $-2-j.$ The result is that
\begin{eqnarray} \label{green5}
G_j(x,y) = r(x,y)^{2-n+j} F_j(x), \,\,\, F_j \in C^{\infty}(U), \,\, j=0,1,2,....
\end{eqnarray}

\begin{definition} \label{con}
 Let $(r(y), \omega(y)) \in \R^+ \times { {\mathbb S}}^{n-1}$ denote polar coordinates in the $y$-variables centered at $x \in U.$ We say that $a \in {\mathcal D}'(U \times U)$ is {\em conormally bounded} if
$$ \sup_{(x,y) \in U \times U}| (r_y \partial_{r_y})^{\alpha} \partial_{\omega_y}^{\beta} a(x,y) |   \leq C_{\alpha, \beta} < \infty.$$

\end{definition}
 
In view of (\ref{locgreen}) and (\ref{green5}) above, we have proved the following

\begin{proposition} \label{hadamard}
When $n \geq 3$ and $U\subset \Omega$ is a sufficiently small open set, the local Green's functions $G(x,y) \in {\mathcal D}'(U \times U)$ has the near-diagonal Hadamard parametrix form
$$ G(x,y) \sim r(x,y)^{2-n} \sum_{j=0}^{\infty} F_j(x)  \, r(x,y)^{j} + R(x,y),$$
in the sense that
$ G(x,y) - r^{2-n}(x,y) \sum_{j=0}^{N} F_j(x) r^j(x,y) = R^{(M)}(x,y) + R(x,y),$ where $ r^{-N-2 +n} R^{(M)}$  is conormally bounded with  $F_j \in C^{\infty}(U)$ for $ j=0,1,2,...$ and $R \in C^{\infty}(U\times U).$
\end{proposition}

For convenience, we denote partial sums in the formal Hadamard series for $G$ by
\begin{equation} \label{M}
G^{(M)}(x,y):= r^{2-n}(x,y) \Big( \sum_{j=0}^M F_j(x) r(x,y)^j \Big).
\end{equation}

Then, by reflection in the boundary,
$$ 2 G_{D}(x,y) = G(x,y) - G(x,y^*), $$
$$2 G_{N}(x,y) = G(x,y) + G(x,y^*)$$ 
where $(x,y) \in U \times U.$ 
Similarily, we denote the respective $M$-th order Hadamard approximations by $2 G^{(M)}_D(x,y) := G^{(M)}(x,y) - G^{(M)}(x,y^*)$ and $2 G^{(M)}_N(x,y):= G^{(M)}(x,y) + G^{(M)}(x,y^*).$

In our analysis, we will need the following immediate corollary of Proposition \ref{hadamard} for $G_{D,N}.$

\begin{proposition} \label{hadamard2}
 When  $(x,y) \in U \times U$ and $U \subset \Omega$ is a sufficiently small open neighbourhood   the local Green's functions $G_{D,N}(x,y) \in {\mathcal D}'(U \times U)$  satisfy
$$ 2 G_{D}(x,y) \sim r(x,y)^{2-n} \sum_{j=0}^{\infty} F_j(x)  \, r(x,y)^{j}  -  r(x,y^*)^{2-n} \sum_{j=0}^{\infty} \tilde{F}_j(x)  \, r(x,y^*)^{j} + R_{D}(x,y)$$

where $R_{D} \in C^{\infty}(U \times U)$ and $F_j, \tilde{F}_j \in C^{\infty}(U).$  In analogy with the free case in Proposition \ref{hadamard}, we denote the truncated sum at $j=M$ by $2 G_{D}^{(M)}(x,y).$ \\

\noindent Similarily,
$$ 2 G_{N}(x,y) \sim r(x,y)^{2-n} \sum_{j=0}^{\infty} F_j(x)  \, r(x,y)^{j} +  r(x,y^*)^{2-n} \sum_{j=0}^{\infty} \tilde{F}_j(x)  \, r(x,y^*)^{j} + R_{N}(x,y),$$
where $R_{N} \in C^{\infty}(U \times U)$ and $F_j, \tilde{F}_j \in C^{\infty}(U).$  We denote the truncated sum at $j=M$  by $2 G_N^{(M)}(x,y).$
\end{proposition}

 Fix  $x \in  \partial \Omega$ and $\phi_h$ be a Neumann eigenfunction.  By an application of Green's formula in $B_h^+(x)$ 
 \begin{align} \label{UPSHOTBDY}
 \phi(x)  = & c_n h^{-2} \int_{B^+_h(x)} \phi(y) G_N(x,y) dy  \nonumber \\
 & + \int_{\partial B_h^+(x)} \partial_{\nu(y)} \phi(y) \cdot G_N(x,y) \, d\sigma(y)  \nonumber \\
&  - \int_{\partial B_h^+(x)}  \phi (y)  \partial_{\nu(y)} G_N(x,y) \, d\sigma(y)
\end{align} \\

From Propositions \ref{hadamard}, \ref{hadamard2} we estimate the first term  on the RHS of (\ref{UPSHOTBDY}) as in (\ref{ibp}) by substituting the Hadamard approximation $G_N^{(M)}.$ 

\begin{eqnarray} \label{intterm}
 h^{-2} \int_{B^+_h(x)} \phi(y)   G_N(x,y)  dy  \hspace{2in}    \nonumber \\
 =  h^{-2} \int_{B_h^+(x)} \phi(y) G_{N}^{(M)}(x,y) dy +   h^{-2} \int_{B_h^+(x)} \phi(y) R_{N}^{(M)}(x,y) dy \nonumber \\
 =   h^{-2} \int_{B_h^+(x)} \phi(y) G^{(M)}(x,y) dy +   h^{-2} \int_{B_h^+(x)} \phi(y) R^{(M)}(x,y) dy.
 \end{eqnarray}
 In the last line of (\ref{intterm}) we have used that when $x \in \partial \Omega, \,\, r(x,y',y_n) = r(x,y',-y_n) $ and so, $G^{(M)}_N(x,y) = G^{(M)}(x,y).$ 
 
For the last remainder term in (\ref{intterm}), choosing $M >n,$

$$  h^{-2} \int_{B_h^+(x)} \phi(y) R^{(M)}(x,y) dy = O(h^{-2}) \int_{B_h^+(x)} \phi(y) O(r^{M+2-n}) dy$$
$$ = O(h^{M-n}) \| \phi_h \|_{B_h^+} |B_h|^{1/2} = O(h^{M-\frac{n}{2}}) \| \phi_h \|_{B_h^+}, $$

which is residual. 

The first integral term in the last line of (\ref{intterm}) is controlled by a radial  integration by parts argument as in (\ref{parts}) with resulting bound

  $$ h^{-2} \int_{B_h^+(x)} \phi(y) G^{(M)}(x,y) dy = O(h^{-n/2}) \| \phi_h \|_{B_{2h}^+(x)},$$
  
  and so
  \begin{equation} \label{inttermbound}
   h^{-2} \int_{B^+_h(x)} \phi(y)   G_N(x,y)  dy = O(h^{-n/2}) \| \phi_h \|_{B_{2h}^+(x)}.
   \end{equation}

To estimate the boundary terms in (\ref{UPSHOTBDY}), we decompose  $\partial B_h^+(x) = ( \partial B_h^+(x) \cap \partial \Omega) \cup S_h^+(x),$
where $S_h^+(x) = \{ y \in \Omega; r(x,y) = h, \, y_n >0\}$ is the hemispherical part of the boundary. 

In the Neumann case, $\partial_{\nu(y)} \phi(y)  = 0$ for all $y \in \partial B_h^+(x) \cap \partial \Omega$  and so,

\begin{equation} \label{bdy1}
\int_{\partial B_h^+(x)} \partial_{\nu(y)} \phi(y) \cdot G_N(x,y) \, d\sigma(y). = \int_{S_h^+(x)} \partial_{\nu(y)} \phi(y) \cdot G_N(x,y) \, d\sigma(y), \quad x \in \partial \Omega.
\end{equation}

In the second boundary integral  on the RHS of  (\ref{UPSHOTBDY}) we use that  $\partial_{\nu(y)} G_{N}(x,y) = 0$ for all $y \in \partial \Omega$ to get that 
\begin{equation} \label{bdy2}
 \int_{\partial B_h^+(x)} \phi(y) \cdot  \partial_{r_y} G_N(x,y) \, d\sigma(y)  = \int_{S_h^+(x)} \phi(y) \cdot  \partial_{r_y} G_{N}(x,y) \, d\sigma(y), \quad x \in \partial \Omega.
 \end{equation}

 We now subsitute the Hadamard approximation $G^{(M)}_N$ in (\ref{UPSHOTBDY}).  Since $x \in \partial \Omega,$ it follows that  $r(x,y) = r(x,y^*)$ for $y \in \partial B_h(x)$ and so, in particular, from Proposition \ref{hadamard2}, 
 $$ G^{(M)}_N(x,y) = G^{(M)}(x,y), \quad x \in \partial \Omega, \,\, y \in \partial B_h^+(x).$$

Consequently,  one can write the boundary contribution in  (\ref{UPSHOTBDY}) in the form
 
 \begin{eqnarray} \label{sub}
  = \int_{S_h^+} \partial_{\nu_y} \phi(y) \cdot G^{(M)}(x,y) \, d\sigma(y)  - \int_{S^+_h} \phi(y) \cdot  \partial_{\nu_y} G^{(M)}(x,y) \, d\sigma(y) \nonumber \\
  + \int_{\partial B_h^+} \partial_{\nu_y} \phi(y) \cdot R^{(M)}(x,y) \, d\sigma(y)  - \int_{\partial B_h^+} \phi(y) \cdot  \partial_{\nu_y} R^{(M)}(x,y) \, d\sigma(y) \nonumber \\
 + \int_{\partial B_h^+} \partial_{\nu(y)} \phi(y) \cdot R(x,y) \, d\sigma(y)  - \int_{\partial B^+_h} \phi(y) \cdot  \partial_{\nu_y} R(x,y) \, d\sigma(y).
 \end{eqnarray}

A key point here is that when $ y \in S_h^+,$
\begin{eqnarray} \label{const1}
G^{(M)}(x,y) = r^{2-n}(x,y)  \Big( \sum_{j=0}^{M} F_j(x) r^{j}(x,y) \Big)  = h^{2-n}  \Big( \sum_{j=0}^M F_j(x) h^j \Big),
\end{eqnarray}

which is {\em constant} as a function of $y \in S_h^+.$

Also, since $\partial_{\nu(y)} = (1 +  f(x,y) ) \partial_r,$ with $f \in C^{\infty},$ direct computation gives with $r= r(x,y),$
\begin{eqnarray} \label{const2}
\partial_{\nu(y)} G^{(M)}(x,y) = r^{1-n}\Big( \sum_{j=0}^{M}  (2-n+j) F_j(x) r^{j}\Big)  \,  ( 1 + f(x,y)) \nonumber \\
= h^{1-n} \Big( \sum_{j=0}^{M}  (2-n+j) F_j(x) h^{j}\Big) \,( 1 +  f(x,y) ) \quad \text{when} \,\, y \in S_h^+,
\end{eqnarray}
with $f, F_j \in C^{\infty}_{loc}$ for all $j=0,1,2,....$ In the following, we set $k(x,y) = ( 1 + f(x,y)) \in C^{\infty}_{loc}$
where $|f(x,y)| = O(r).$

Since $\partial_{\nu(y)} \phi_h(y) = 0$  for $ y \in \partial \Omega,$  in view of (\ref{const1}), 

\begin{eqnarray} \label{N1}
\int_{S_h^+(x)} \partial_{\nu(y)} \phi(y) \cdot G^{(M)}(x,y) \, d\sigma(y) = h^{2-n}  \, \big( \sum_{j=0}^M F_j(x) h^j  \big)\, \int_{ \partial B_h^+(x)} \partial_{\nu} \phi(y) \, d\sigma(y) \nonumber\\
= O(h^{2-n}))  \int_{B_h^+} \Delta_y \phi(y) dy = O(h^{-n}) |B_h|^{1/2} \| \phi \|_{B_h^+}= O(h^{-n/2}) \| \phi \|_{B_h^+},  \end{eqnarray}
where the last  line of (\ref{N1}) follow from Green's formula and Cauchy-Schwarz.

For the second integral in the first line of  (\ref{sub}),
\begin{eqnarray} \label{N2}
\int_{S^+_h(x)} \phi(y) \cdot  \partial_{\nu_y} G^{(M)}(x,y) \, d\sigma(y) \hspace{2in} \nonumber \\
= \frac{h^{-n}}{2} \Big( \sum_{j=0}^{M}  (2-n+j) F_j(x) h^{j}\Big) \, \int_{S_h^+(x)} \phi(y) k(x,y)  \partial_r (r^2) \, d\sigma(y) \nonumber \\
= \frac{h^{-n}}{2} \big( \sum_{j=0}^{M}  (2-n+j) F_j(x) h^{j}\big) \, \Big( \, \int_{\partial B_h^+(x)} \phi(y) \partial_{\nu_y}  (r^2) \, d\sigma(y) -  \int_{\partial B_h^+(x)}  \partial_{\nu_y} \phi (y) r^2 d\sigma \Big) \nonumber \\
+\frac{h^{-n}}{2} \big( \sum_{j=0}^{M}  (2-n+j) F_j(x) h^{j}\big) \int_{\partial B_h^+(x)}  \partial_{\nu_y} \phi (y) r^2 d\sigma.
\end{eqnarray}

In the penultimate line of (\ref{N2}), we have used that
$$ \partial_{\nu_y} r^2 |_{y_n=0} = \partial_{y_n} ( |x'-y'|^2 + y_n^2 ) |_{y_n = 0} = 0,$$

to write the integral over $S_h^+$  as an integral over $\partial B^+_h$ by adding a null term.

We estimate the last integral in (\ref{N2}). Since $\partial_{\nu_y}\phi |_{\partial \Omega} = 0$ and $r=h$ on $S_h^+,$ it follows that
\begin{eqnarray} \label{N3}
\int_{\partial B_h^+(x)}  \partial_{\nu_y} \phi (y) r^2 d\sigma  = h^2 \int_{\partial B_h^+(x)}  \partial_{\nu_y} \phi (y) d\sigma \nonumber \\
= h^2 \int_{B_h^+} \Delta_y \phi \, dy = O(1) \| \phi \|_{B_h^+} |B_h^+|^{1/2}.
\end{eqnarray}
Consequently, the last integral in (\ref{N2}) is $O(h^{-n/2}) \| \phi \|_{B_h^+}$ as required.

To estimate the penultimate line in (\ref{N2}), we apply Green yet again to write it in  the form

\begin{eqnarray} \label{N4}
\frac{h^{-n}}{2} \big( \sum_{j=0}^{M}  (2-n+j) F_j(x) h^{j}\big) \, \Big( \, \int_{B_h^+(x)}  \phi(y) \Delta_y  (r^2) \, dy  -  \int_{B_h^+(x)}  \Delta_y \phi (y) r^2 d\sigma \Big) \nonumber\\
=O(h^{-n}) \| \phi \|_{B_h^+} |B_h^+|^{1/2} = O(h^{-n/2} \|\phi \|_{B_h^+}).
\end{eqnarray}

Finally, we estimate the remainder terms in (\ref{sub}) by another application of Green to get:

\begin{eqnarray} \label{N5}
 \int_{\partial B_h^+} \partial_{\nu_y} \phi(y) \cdot R^{(M)}(x,y) \, d\sigma(y)  - \int_{\partial B_h^+} \phi(y) \cdot  \partial_{\nu_y} R^{(M)}(x,y) \, d\sigma(y) \nonumber \\
 =  \int_{B_h^+} \Delta_y \phi(y) \cdot R^{(M)}(x,y) \, dy  - \int_{ B^+_h} \phi(y) \cdot  \Delta_yR^{(M)}(x,y) dy.
 \end{eqnarray}
  
 Since $R^{(M)} = O(r^{2-n+M})$ and $\Delta_y R^{(M)} = O(r^{M -n}),$ we choose $M >n$ and it follows that (\ref{N5}) is
 $$ O(h^{-2} h^{M-n+2} |B_h|^{1/2} \| \phi \|_{B_h^+})= O(h^{-n/2 +M})\| \phi \|_{B_n^+},$$
 which is residual. 
 
 Finally for the last remainder term, $R \in C^{\infty}(U \times U)$ and just use Green and the obvious bounds $\| R \|_{\infty} =O(1)$ and $\| \Delta_y R \|_{\infty} = O(1)$ to get that
 
 \begin{eqnarray} \label{N6}
 \int_{\partial B_h^+} \partial_{\nu_y} \phi(y) \cdot R(x,y) \, d\sigma(y)  - \int_{\partial B^+_h} \phi(y) \cdot  \partial_{\nu_y} R(x,y) \, d\sigma(y) \nonumber \\
 =  \int_{B_h^+} \Delta_y \phi(y) \cdot R(x,y) \, dy  - \int_{ B^+_h} \phi(y) \cdot  \Delta_y R(x,y) dy \nonumber \\
 = O(h^{-2} |B_h|^{1/2}) \| \phi \|_{B_h^+} =O(h^{n/2 - 2} ) \| \phi \|_{B_h^+}.
 \end{eqnarray}  
Since $n/2 - 2 > 1/2 -n/2$ for all $n \geq 3$, this term is also residual. 

Since all estimates above are uniform in $x \in \partial \Omega,$  in view of (\ref{N1})-(\ref{N6}) we have proved here that
\begin{equation} \label{bdyestimate}
 \sup_{x \in \partial \Omega} | \phi_h(x)| = O(h^{-\frac{n}{2}}) \| \phi_h \|_{B_h^+(x)}.
 \end{equation}

\subsection{Estimates in the boundary layer.}

 Consider the boundary layer $\partial \Omega(8^{-1} h) = \{ 0 \leq x_n \leq \frac{ h}{8} \}$ where $(x',x_n): \Omega(\epsilon) \to \R^{n}$ continue to denote Fermi coordinates where $\epsilon >0$ is sufficiently small. Then, $\partial \, ( \partial \Omega(\frac{ h}{8} ) ) = \{ x_n = 8^{-1} h \} \cup \{ x_n = 0 \}$ where $\partial \Omega= \{ x_n = 0 \}.$ To estimate $\| \phi_h \|_{ L^\infty( \partial \Omega( \frac{ h}{8} )},$ we follow an argument similar to that in  \cite{Gr} (see Lemma 6) and apply an adapted maximum principle argument. Set
 
 $$ \psi_h(x) = \exp (2 h^{-1}x_n) \phi_h(x), \quad x \in \partial \Omega( 8^{-1}  h).$$
 
 A straightforward computation with the Laplacian in Fermi coordinates shows that
 $$ P \psi_h(x) = 0, \quad x \in \partial \Omega( 8^{-1}  h ),$$
 where $ P = - \sum_{i,j} g^{ij}(x) \partial_i \partial_j + \sum_i a_i(x) \partial_i + c$ where 
 \[
 c = \frac{3}{h^2} - \frac{4}{h} >0
 \]
 for $h$ sufficiently small.
 
Then, by the weak maximum principle applied to $P$ and $\psi,$ it follows that
 
 \begin{equation} \label{layerbd}
\| \psi_h \|_{ L^\infty( \partial \Omega(8^{-1} h )} ) =  \max_{ \{ x_n = 8^{-1}  h \}} |\psi_h(x)|  + \max_{ \{ x_n = 0 \}} |\psi_h(x)|. \end{equation}

Since  $|\psi_h(x)| = | \exp(2 x_n h^{-1})| \cdot |\phi_h(x)|,$ we have 
\[
 C^{-1} | \phi_h(x) | \leq | \psi_h(x) |\leq  C |\phi_h(x)|,
 \]
  for $C>0$ independent of $h$ in $\Omega (8^{-1} h)$, it follows from the interior sup bounds (\ref{intsupbounds}) that
$$ \max_{ \{ x_n = 8^{-1}  h \} }  |\psi_h(x)| = O(h^{-\frac{n}{2}}\sup_{ \{ x_n = 8^{-1}  h \} } \| \phi \|_{L^2(B_h(x))} )$$
and from the boundary estimate (\ref{bdyestimate})
$$ \max_{ \{ x_n = 0 \} }  |\psi_h(x)| = O(h^{-\frac{n}{2}} \max_{x\in \partial \Omega} \| \phi \|_{L^2(B_h^+(x))}   )$$

Then, from (\ref{layerbd}),

\begin{equation} \label{bdylayerupshot}
\| \psi_h \|_{ L^\infty( \partial \Omega( 8^{-1}h ) )} = O(h^{\frac{1-n}{2}} ).
\end{equation}

The boundary layer estimate

\begin{equation} \label{bdylayerupshot2}
\| \phi_h \|_{ L^\infty(\partial \Omega( 8^{-1} h ) )} = O(h^{-\frac{n}{2}}  \sup_{x \in \partial \Omega( 8^{-1}  h )) }  \| \phi \|_{L^2(B_h^+(x))}  )
\end{equation}
then follows from (\ref{layerbd}), since $C^{-1} \leq \exp ( 2 x_n/h) \leq C$, $C>0$ independent of $h$, when $|x_n| \leq 8^{-1}/h .$

This completes the proof of Theorem \ref{thm2}. \end{proof}

\begin{remark}  (a) As an immediate corollary of Theorems \ref{thm1} and \ref{thm2} one obtains the sup bound $\| \phi_h \|_{L^\infty(\Omega)} = O(h^{\frac{1-n}{2} })$ for Dirichlet or Neumann eigenfunctions when $\Omega$ is a compact $C^{\infty}$ Riemannian manifold with boundary. We note that  unlike the argument in \cite{Gr}, the proof here of the sharp eigenfunction sup bounds is entirely stationary and local using only the non-concentration bound up to the  boundary  in Theorem \ref{thm1} and the local integral estimate in Theorem \ref{thm2}.

(b) In view of Theorems 1 and 2,  any improvement in the non-concentration bound for $\| \phi_h \|_{L^2(B_h)(x)}$ for all $x \in \Omega$ as in the quantum ergodic (QE) case in  \cite{Han, HR}, automatically implies the corresponding improvement in the upper bound for $ \| \phi_h \|_{L^\infty(\Omega)}$ all the way up to the boundary $\partial \Omega$ away from corners.

\end{remark}

\end{document}